\documentclass[11pt]{article}
\usepackage[english]{babel}
\usepackage[a4paper,left=3cm, right=3cm,top=3.5cm, bottom=4.5cm]{geometry}
\usepackage{amsmath,amssymb,amsfonts, marvosym}
\usepackage{epsfig}
\usepackage{tikz}
\usepackage{graphics}
\usepackage{bbm}
\usepackage{bm, bbm}
\usepackage{subcaption}
\usepackage{boxedminipage,verbatim,float,caption}
\usepackage{algorithm}
\usepackage{algpseudocode}
\usepackage{todonotes}
\usepackage{complexity}
\usepackage[thmmarks]{ntheorem}
\usepackage[]{hyperref}
 \hypersetup{colorlinks=true, citecolor=black,%
                 filecolor=black,%
                 linkcolor=black,%
                 urlcolor=black,  pdftex}

\usetikzlibrary[graphs, arrows, backgrounds, intersections, positioning, fit, petri, calc, shapes, decorations.pathmorphing]
\usetikzlibrary{arrows.meta}
\usepgflibrary{patterns}
\tikzset{node distance=1cm, bend angle=20,
vertex/.style={circle,minimum size=2mm,very thick, draw=black, fill=black, inner sep=0mm}, information text/.style={inner sep=1ex, font=\Large}, help lines/.style={-,color=black, >=stealth', shorten <=.5pt, shorten >=.5pt}, blue help lines/.style={help lines,color=darkblue}, red help lines/.style={help lines, color=darkred},
>={[scale=1.1]Stealth}}

\definecolor{darkred}{RGB}{105,0,0}
\usepackage[retainorgcmds]{IEEEtrantools}

\makeatletter
\newtheoremstyle{prime}%
 {\item[\hskip\labelsep \theorem@headerfont ##1\ \theorem@separator]}%
{\item[\hskip\labelsep \theorem@headerfont ##1\ ##3' \theorem@separator]}
\makeatother

\makeatletter
\newtheoremstyle{proofof}
{\item[\hskip\labelsep \theorem@headerfont ##1\ \theorem@separator]}%
{\item[\hskip\labelsep \theorem@headerfont ##1\ ##3\theorem@separator]}
\makeatother

\newtheorem{theorem}{Theorem}

\newtheorem{proposition}[theorem]{Proposition}
\newtheorem{observation}[theorem]{Observation}
\newtheorem{corollary}[theorem]{Corollary}
\newtheorem{conjecture}{Conjecture}
\newtheorem{question}[conjecture]{Question}

\newtheorem{claim}{Claim}

\theoremstyle{prime}

\def \QD1 {\hfill $\spadesuit$}

\newcommand{\case}[2]{\noindent {\bf Case #1\/:} {\it #2}}

\newcommand{\bd}[1]{\overset{\text{\tiny$\bm\leftrightarrow$}}{#1}}
\newcommand{\overr}[1]{\overset{\text{\tiny$\bm\rightarrow$}}{#1}}
\newcommand{\overl}[1]{\overset{\text{\tiny$\bm\leftarrow$}}{#1}}

\newcommand{\ems}{\varnothing}
\newcommand{\sm}{\setminus}

\numberwithin{equation}{section}

\theoremstyle {nonumberplain}
\theoremseparator{:}
\theorembodyfont{\normalfont}
\theoremsymbol{\rule{1ex}{1ex}}
\newtheorem{proof}{Proof}

\theoremstyle{proofof}
\newtheorem{proofof}{Proof of}

\theoremsymbol{$\square$}
\newtheorem{proof2}{Proof}

\theoremsymbol{$\diamond$}

\newcommand{\jou}[4]{{\em #1} {\bf #2} (#3) #4.}
\def \JCTB {J. Combin. Theory \, Ser.~B}
\def \DM {Discrete Math.}

\begin{document}
\title{\bf Haj\'os and Ore constructions for digraphs}

\author{{{J{\o}rgen Bang-Jensen}\thanks{Research supported by the Danish
research council under grant number 7014-00037B}
\thanks{
University of Southern Denmark, IMADA, Campusvej 55, DK-5320 Odense M, Denmark. E-mail address: jbj@imada.sdu.dk
}}
\and
{{Thomas Bellitto}\footnotemark[1]
\thanks{
University of Southern Denmark, IMADA, Campusvej 55, DK-5320 Odense M, Denmark. E-mail address: bellitto@imada.sdu.dk
}}
\and
{{Thomas Schweser}\footnotemark[1]
\thanks{
Technische Universit\"at Ilmenau, Inst. of Math., PF 100565, D-98684 Ilmenau, Germany. E-mail
address: thomas.schweser@tu-ilmenau.de}}
\and
{{Michael Stiebitz}\thanks{
Technische Universit\"at Ilmenau, Inst. of Math., PF 100565, D-98684 Ilmenau, Germany. E-mail
address: michael.stiebitz@tu-ilmenau.de}}
}

\date{}
\maketitle

\begin{abstract}
The chromatic number $\overr{\chi}(D)$ of a digraph $D$ is the minimum number of colors needed to color the vertices of $D$ such that each color class induces an acyclic subdigraph of $D$. A digraph $D$ is $k$-critical if $\overr{\chi}(D) = k$ but $\overr{\chi}(D') < k$ for all proper subdigraphs $D'$ of $D$. We examine methods for creating infinite families of critical digraphs, the \emph{Dirac join} and the directed and bidirected \emph{Haj\'os join}. We prove that a digraph $D$ has chromatic number at least $k$ if and only if it contains a subdigraph that can be obtained from bidirected complete graphs on $k$ vertices by (directed) Haj\'os joins and identifying non-adjacent vertices. Building upon that, we show that a digraph $D$ has chromatic number at least $k$ if and only if it can be constructed from bidirected $K_k$'s by using directed and bidirected Haj\'os joins and identifying non-adjacent vertices (so called Ore joins), thereby transferring a well-known result of Urquhart to digraphs. Finally, we prove a Gallai-type theorem that characterizes the structure of the low vertex subdigraph of a critical digraph, that is, the subdigraph, which is induced by the vertices that have in-degree $k-1$ and out-degree $k-1$ in $D$.
\end{abstract}



\section{Introduction}

Recall that the \textbf{chromatic number} $\chi(G)$ of a graph $G$ is the minimum number of colors needed to color the vertices of $G$ so that each color class induces an edgeless subgraph of $G$. A graph $G$ is $k$-critical if $\chi(G)=k$ but $\chi(G') < k$ for each proper subgraph $G'$ of $G$. The topic of critical graphs has received much attention within the last century. Critical graphs were first introduced by G.~A.~Dirac in his doctoral thesis; famous mathematicians like G~Haj\'os, T.~Gallai and others continued developing the theory of critical graphs in the 1960's.

However, not much is known regarding critical digraphs. Following Neumann-Lara~\cite{NeuLa82}, the chromatic number $\overr{\chi}(D)$ of a digraph $D$ is the minimum number of colors needed to color the vertices of $D$ so that each color class induces an acyclic subdigraph of $D$, \emph{i.e.}, a subdigraph that does not contain any directed cycles. A digraph $D$ is $k$\textbf{-critical} (or, briefly, \textbf{critical}) if $\overr{\chi}(D)=k$ but $\overr{\chi}(D')<k$ for each proper subdigraph of $D$. Let $\text{Crit}(k)$ denote the class of $k$-critical digraphs. Then, it is easy to see that $\text{Crit}(0)=\{\varnothing\}, \text{ Crit}(1)=\{K_1\}$, and that $\text{Crit}(2)$ consists of all directed cycles. Nevertheless, it is not even known which digraphs $\text{Crit}(3)$ consists of; unlike in the undirected case, where it follows from König's characterization of bipartite graphs that $\text{Crit}(3)$ coincides with the class of all odd cycles. In this paper, we study the digraph analogue of two well-known methods for creating infinite families of critical graphs, the so-called \emph{Dirac join} and the \emph{Haj\'os join}. Moreover, we prove that a digraph $D$  has chromatic number at least $k$ if and only if it contains a Haj\'os-$k$-constructible subdigraph, that is, a subdigraph of $D$ that can be obtained from bidirected $K_k$'s by iteratively applying the Haj\'os join and identifying non-adjacent vertices (see Theorem~\ref{theorem_Hajos-constructible}). In Section~\ref{section_gallai} we prove a Gallai-type theorem that characterizes the structure of the low-vertex subdigraph of a $k$-critical digraph, that is, the subdigraph that is induced by the vertices having in-degree and out-degree $k-1$.

\section{Basic Terminology}
Most of our terminology is defined as in \cite{BaGu08}. Let $D=(V(D),A(D))$ be a digraph. Then, $V(D)$ is the \textbf{set of vertices} of $D$ and $A(D)$ is the \textbf{set of arcs} of $D$. The order $|D|$ of $D$ is the size of $V(D)$. Digraphs in this paper are not allowed to have loops nor parallel arcs; however, there may be two arcs in opposite directions between two vertices (in this case we say that the arcs are \textbf{opposite}). We denote by $uv$ the arc whose \textbf{initial vertex} is $u$ and whose \textbf{terminal vertex} is $v$.
Two vertices $u,v$ are \textbf{adjacent} if at least one of $uv$ and $vu$ belongs to $A(D)$. If $u$ and $v$ are adjacent, we also say that $u$ is a \textbf{neighbor} of $v$ and vice versa. If $uv \in A(D)$, then $v$ is called an \textbf{out-neighbor} of $u$ and $u$ is called an \textbf{in-neighbor} of $v$. By $N_D^+(v)$ we denote the set of out-neighbors of $v$; by $N_D^-(v)$ the set of in-neighbors of $v$. Given a digraph $D$ and a vertex set $X$, by $D[X]$ we denote the subdigraph of $D$ that is \textbf{induced} by the vertex set $X$, that is, $V(D[X])=X$ and $A(D[X])=\{uv \in A(D) ~|~ u,v \in X\}$. A digraph $D'$ is said to be an induced subdigraph of $D$ if $D'=D[V(D')]$. As usual, if $X$ is a subset of $V(D)$, we define $D-X=D[V(D) \setminus X]$. If $X=\{v\}$ is a singleton, we use $D-v$ rather than $D- \{v\}$. The \textbf{out-degree} of a vertex $v \in V(D)$ is the number of arcs whose inital vertex is $v$; we denote it by $d_D^+(v)$. Similarly, the number of arcs whose terminal vertex is $v$ is called the \textbf{in-degree} of $v$ and is denoted by $d_D^-(v)$. Note that $d_D^+(v)=|N_D^+(v)|$ and $d_D^-(v)=|N_D^-(v)|$ for all $v \in V(D)$. A vertex $v \in V(D)$ is \textbf{Eulerian} if $d_D^+(v)=d_D^-(v)$. Moreover, the digraph $D$ is \textbf{Eulerian} if every vertex of $D$ is Eulerian.

Given a digraph $D$, its \textbf{complement} is the digraph $\overline{D}$ with $V(\overline{D})=V(D)$ and $A(\overline{D})=\{uv ~|~ u,v \in V(D) \text{ and } uv \not \in A(D)\}$. The \textbf{underlying} graph $G(D)$ of $D$ is the simple undirected graph with $V(G(D))=V(D)$ and $\{u,v\}\in E(G(D))$ if and only if at least one of $uv$ and $vu$ belongs to $A(D)$. The digraph $D$ is \textbf{(weakly) connected} if $G(D)$ is connected. A \textbf{separating vertex} of a connected digraph $D$ is a vertex $v \in V(D)$ such that $D-v$ is not connected. Furthermore, a \textbf{block} of $D$ is a maximal subdigraph $D'$ of $D$ such that $D'$ has no separating vertex.

A \textbf{directed path} is a non-empty digraph $P$ with $V(P)=\{v_1,v_2,\ldots,v_p\}$ and $A(P)=\{v_1v_2, v_2v_3, \ldots, v_{p-1}v_p\}$ where the $v_i$ are all distinct. Furthermore, a \textbf{directed cycle} of  \textbf{length} $p\geq 2$ is a non-empty digraph $C$ with $V(C)=\{v_1,v_2,\ldots,v_p\}$ and $A(C)=\{v_1v_2,v_2v_3, \ldots, v_{p-1}v_p, v_pv_1\}$ where the $v_i$ are all distinct. A directed cycle of length $2$ is called a \textbf{digon}.
A \textbf{bidirected} graph is a digraph that can be obtained from an undirected  (simple) graph $G$ by replacing each edge by two opposite arcs, we denote it by $D(G)$. A bidirected complete graph is also called a \textbf{complete digraph}. By $\bd{K_k}$ we denote the bidirected complete graph on $k$ vertices. It is easy to see that $\overr{\chi}(D(G))=\chi(G)$ and $D(G)$ is critical with respect to $\overr{\chi}$ if and only if $G$ is critical with respect to $\chi$.

 \section{Construction of critical digraphs}\label{section_construction}
Let $D_1$ and $D_2$ be two disjoint digraphs. Let $D$ be the digraph obtained from the union $D_1 \cup D_2$ by adding all possible arcs in both directions between $D_1$ and $D_2$, \emph{i.e.}, $V(D)=V(D_1) \cup V(D_2)$ and $A(D)=A(D_1) \cup A(D_2) \cup \{uv, vu ~|~ u \in V(D_1) \text{ and } v \in V(D_2)\}$. We say that $D$ is the \textbf{Dirac join} of $G_1$ and $G_2$ and denote it by $D= D_1 \boxplus D_2$. The proof of the next theorem is straightforward and therefore left to the reader.

\begin{theorem}[Dirac Construction]
Let $D=D_1 \boxplus D_2$ be the Dirac join of two disjoint non-empty digraphs $D_1$ and $D_2$. Then, $\overr{\chi}(D) = \overr{\chi}(D_1) + \overr{\chi}(D_2)$ and $D$ is critical if and only if both $D_1$ and $D_2$ are critical.
\end{theorem}

The Haj\'os join is a well-known tool for undirected graphs that can be used to create infinite families of $k$-critical graphs, see e.~g.~\cite{Ha61}. For digraphs, an equivalent construction was defined by Hoshino and Kawarabayashi in \cite{HoKa14}.
Let $D_1$ and $D_2$ be two disjoint digraphs and select an arc $u_1v_1$ and an arc $v_2u_2$. Let $D$ be the digraph obtained from the union $D_1 \cup D_2$ by deleting the arcs $u_1v_1$ as well as $v_2u_2$, identifying the vertices $v_1$ and $v_2$ to a new vertex $v$, and adding the arc $u_1u_2$. We say that $D$ is the \textbf{(directed) Haj\'os join} of $D_1$ and $D_2$ and write $D=(D_1,v_1,u_1) \triangledown (D_2,v_2,u_2)$ or, briefly, $D=D_1 \triangledown D_2$. For the proof of the next theorem, recall that a $k$-coloring of a digraph $D$ is a coloring of $D$, in which at most $k$ colors are used. Statement (c) of the following theorem has already been mentioned in \cite[Prop. 2]{HoKa14}

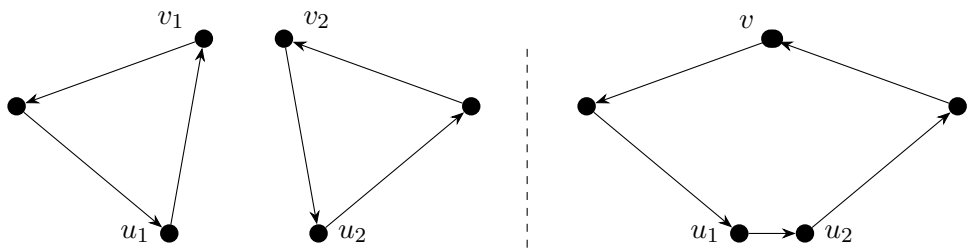
\begin{figure}[htbp]
\centering
\begin{tikzpicture}[>={[scale=1.1]Stealth}]

\node[draw=none,minimum size=3cm,regular polygon,regular polygon sides=3, rotate=-40] (a) {};

\node[vertex, label={177:$v_1$}] (a1) at (a.corner 1) {};
\node[vertex] (a2) at (a.corner 2){};
\node[vertex, label={left:$u_1$}] (a3) at (a.corner 3){};

\node[draw=none,minimum size=3cm,regular polygon,regular polygon sides=3, xshift=3cm, rotate=40] (b) {};
\node[vertex, label={3:$v_2$}] (b1) at (b.corner 1) {};
\node[vertex, label={right:$u_2$}] (b2) at (b.corner 2){};
\node[vertex] (b3) at (b.corner 3){};

\path[-]
(a1) edge [->] (a2)
(a2) edge [->] (a3)
(a3) edge [->] (a1)
(b1) edge [->] (b2)
(b2) edge [->] (b3)
(b3) edge [->] (b1);

\node[xshift=3.2cm] (h1) at (b1){};
\node[xshift=3.2cm, yshift=-3cm] (h2) at (b1){};
\draw[dashed] (h1) -- (h2);

\begin{scope}[xshift=7.5cm]
\node[draw=none,minimum size=3cm,regular polygon,regular polygon sides=3, rotate=-40] (a) {};
\node[vertex, label={177:$v$}] (a1) at (a.corner 1) {};
\node[vertex] (a2) at (a.corner 2){};
\node[vertex, label={left:$u_1$}] (a3) at (a.corner 3){};

\node[draw=none,minimum size=3cm,regular polygon,regular polygon sides=3, xshift=1.9cm, rotate=40] (b) {};
\node[vertex] (b1) at (b.corner 1) {};
\node[vertex, label={right:$u_2$}] (b2) at (b.corner 2){};
\node[vertex] (b3) at (b.corner 3){};

\path[-]
(a1) edge [->] (a2)
(a2) edge [->] (a3)
(a3) edge [->] (b2)
(b2) edge [->] (b3)
(b3) edge [->] (b1);

\end{scope}

\end{tikzpicture}
\caption{The Haj\'{o}s join of two directed cycles of length $3$}
\label{fig_C3Haj}
\end{figure}

\begin{theorem}[Haj\'os Construction] \label{theorem_Hajos-join}
Let $D=D_1 \triangledown D_2$ be the Haj\'os join of two disjoint non-empty digraphs $D_1$ and $D_2$. Then, the following statements hold:
\begin{itemize}
\item[\upshape (a)] $\overr{\chi}(D) \geq \min\{\overr{\chi} (D_1), \overr{\chi}(D_2)\}$.
\item[\upshape (b)] If $\overr{\chi}(D_1)=\overr{\chi}(D_2)=k$ and $k \geq 3$, then $\overr{\chi}(D)=k$.
\item[\upshape (c)] If both $D_1$ and $D_2$ are $k$-critical and $k \geq 3$, then $D$ is $k$-critical.
\item[\upshape (d)] If $D$ is $k$-critical and $k \geq 3$, then both $D_1$ and $D_2$ are $k$-critical.
\end{itemize}
\end{theorem}

\begin{proof} Suppose that $D=(D_1,v_1,u_1) \triangledown (D_2,v_2,u_2)$ and let $v$ denote the vertex that is obtained from identifying $v_1$ and $v_2$. For the proof of (a) let $\overr{\chi}(D)=k$ and let $\varphi$ be a $k$-coloring of $D$. For $i \in \{1,2\}$, let $\varphi_i$ denote the restriction of $\varphi$ to $D_i$, where $\varphi_i(v_i)=\varphi(v).$ We claim that either $\varphi_1$ is a $k$-coloring of $D_1$ or $\varphi_2$ is a $k$-coloring of $D_2$. Otherwise, in $D_1$ there is a monochromatic directed cycle $C_1$ that contains the arc $u_1v_1$ (as $D_1 - u_1v_1$ is a subdigraph of $D$ and therefore $k$-colorable). Similar, in $D_2$ there exists a monochromatic cycle $C_2$ that contains the arc $v_2u_2$. But then, $C_1 \cup C_2 - u_1v_1 - v_2u_2 + u_1u_2$ is a monochromatic directed cycle in $D$, a contradiction. This proves (a).

In order to prove (b), let $\overr{\chi}(D_1)=\overr{\chi}(D_2)=k$. By (a), $\overr{\chi}(D) \geq k$. Thus, it suffices to show that $\overr{\chi}(D) \leq k$. For $i \in \{1,2\}$, let $\varphi_i$ be a $k$-coloring of $D_i$. By permuting the colors if necessary we obtain $\varphi_1(v_1)=\varphi_2(v_2)$. For $w \in V(D)$ let
$$\varphi(w) =
\begin{cases}
\varphi_1(w) \quad \text{if } w \in V(D_1),\\
\varphi_2(w) \quad \text{if } w \in V(D_2), \text{ and  }\\
\varphi_1(v_1) \quad \text{if } w=v.
\end{cases}$$
We claim that $\varphi$ is a $k$-coloring of $D$. For otherwise, $D$ would contain a monochromatic directed cycle $C$ with $\{u_1,u_2,v\} \subseteq V(C)$ and $u_1u_2 \in A(C)$.  But then, $(C \cap D_1) + u_1v_1$ is a monochromatic directed cycle in $D_1$, which is impossible.

For the proof of (c) it suffices to show that $\overr{\chi}(D-a) < k$ for all $a \in A(D)$ (by (b)). If $a = u_1u_2$, then choosing $(k-1)$-colorings of $D_1 - u_1v_1$ and $D_2 - v_2u_2$ that assign the same color to $v_1$ and $v_2$ and taking the union of those colorings clearly leads to a $(k-1)$-coloring of $D$. Let $a \in A(D) \setminus \{u_1u_2\}$. By symmetry, we may assume that $a \in A(D_1)$. Then, there is a $(k-1)$-coloring $\varphi_1$ of $D_1 - a$ and a $(k-1)$-coloring $\varphi_2$ of $D_2 - v_2u_2$ such that $\varphi_1(v_1)=\varphi_2(v_2)$. By taking the union of those colorings we obtain a $(k-1)$-coloring of $D$ and so $D$ is $k$-critical, as claimed.

To prove statement (d) first assume that $\overr{\chi}(D)=k$ but $\overr{\chi}(D_1)=k-1$. Then there is a $(k-1)$-coloring $\varphi$ of $D_1$. Since $D$ is $k$-critical, there furthermore exists a $(k-1)$-coloring $\varphi_2$ of $D_2 - v_2u_2$ with $\varphi_2(v_2)=\varphi_1(v_1)$ and the union of $\varphi_1$ and $\varphi_2$ is a $(k-1)$-coloring of $D$. Hence, $\overr{\chi}(D_1) \geq k$ and, by symmetry, we obtain $\overr{\chi}(D_2) \geq k$. In order to complete the proof we need to show that $\overr{\chi}(D_i - a ) < k$ for $i \in \{1,2\}$ and for $a \in A(D_i)$. By symmetry, it suffices to show this for $D_1$. If $a = u_1v_1$, then $\overr{\chi}(D_1 - a) < k$ as $D_1 - a$ is a proper subdigraph of $D$ and therefore $(k-1)$-colorable. Let $a \in A(D_1) \setminus \{u_1v_1\}$. Then, there is a $(k-1)$-coloring $\varphi$ of $D - a$. We claim that the restriction of $\varphi$ to $V(D_1)$ is a $(k-1)$-coloring of $D_1 - a$. For otherwise, in $D_1 - a$ there would exist a monochromatic directed cycle $C_1$ that contains the arc $u_1v_1$. Since $\overr{\chi}(D_2) \geq k$, the restriction of $\varphi$ to $V(D_2)$ creates a monochromatic directed cycle $C_2$ in $D_2$ that contains the arc $u_2v_2$. However, $C_1 \cup C_2 - u_1v_1 - v_2u_2 + u_1v_1$ is a monochromatic directed cycle in $D - a$ with respect to $\varphi$, a contradiction. This completes the proof.
\end{proof}

Another common operation for graphs and digraphs is the identification of independent sets. Let $D$ be a digraph and let $I$ be a non-empty independent set of $D$. Then, we can create a new digraph $H$ from $D-I$ by adding a new vertex $v=v_I$ and adding all arcs from $v$ to $N_H^+(I)= \bigcup_{v \in I} N_H^+(v)$ and all arcs from $N_H^-(I)=\bigcup_{v \in I}N_H^-(v)$ to $v$. We say that $H$ is obtained from $D$ by \textbf{identifying $I$ with $v$}, or briefly by \textbf{identifying independent vertices} and write $H=D/(I \to v)$ (briefly $H=D/I$). It is obvious that any $k$-coloring of $D/I$ can be extended to a $k$-coloring of $D$ by coloring each vertex of $I$ with the color of $v_I$. Thus, $\overr{\chi}(D/I) \geq \overr{\chi}(D)$.

We define the class of \textbf{Haj\'os-$k$-constructible digraphs} as the smallest family of digraphs that contains the bidirected complete graph of order $k$ and is closed ander Haj\'os joins and identifying independent vertices. The next result was proved for undirected graphs by Haj\'os \cite{Ha61}.

\begin{theorem}\label{theorem_Hajos-constructible}
Let $k \geq 3$ be an integer. A digraph has chromatic number at least $k$ if and only if it contains a Haj\'os-$k$-constructible subdigraph.
\end{theorem}

For the proof of the above theorem we need a result on perfect digraphs. Recall that the \textbf{clique number} $\omega(D)$ of a digraph $D$ is the size of the largest bidirected complete subdigraph of $D$. A \textbf{perfect digraph} is a digraph $D$ satisfying that for each induced subdigraph $H$ of $D$ it holds $\overr{\chi}(H)=\omega(H)$. Recall that an \textbf{odd hole} is an (undirected) cycle of odd length at least $5$  and an \textbf{odd antihole} is the complement of an odd hole. Moreover, a \textbf{filled odd hole/antihole} is a digraph $D$ so that $S(D)$ is an odd hole/antihole, where $S(D)$ is the \textbf{symmetric part} of $D$, that is, the graph with vertex set $V(D)$ and edge set $$E(S(D))=\{uv ~|~ uv \in A(D) \text{ and } vu \in A(D)\}.$$ Andres and Hochstättler \cite[Corollary~5]{AnHo15} proved the following result on perfect digraphs.


\begin{theorem}[Andres and Hochstättler]\label{theorem_perfect-digraph}
A digraph $D$ is perfect if and only if it contains none of the following as an induced subdigraph: a filled odd hole, a filled odd antihole, or a directed cycle of length at least $3$ as induced subdigraph.
\end{theorem}

This theorem is a really nice and powerful tool in many ways. If $D=D(G)$ is a bidirected graph, then the theorem is equivalent to the Strong Perfect Graph Theorem (SPGT) by Chudnovsky, Robertson, Seymour, and Thomas~\cite{CRST06}, and hence, the SPGT follows from Andres and Hochstättler's result. Nevertheless, their proof heavily relies on the SPGT.

\begin{proofof}[Theorem \ref{theorem_Hajos-constructible}] Let $k\geq 3$ be an integer. Clearly, every Haj\'os-$k$-constructible digraph has chromatic number at least $k$ (by Theorem~\ref{theorem_Hajos-join} and since $\overr{\chi}(D/I) \geq \overr{\chi}(D)$ for each independent set $I$ of a digraph $D$). This proves the ``if''-implication. The proof of the ``only if''-implication is by reductio ad absurdum. Let $D$ be a maximal counter-example in the sense that $D$ does not contain a Haj\'os-$k$-constructible  subdigraph but adding a new arc $a \in A(\overline{D})$ to $D$ implies the existence of a Haj\'os-$k$-constructible subdigraph $D_a$ of $D + a$ with $a \in A(D_a)$. As $D$ is not Haj\'os-$k$-constructible, $D$ is not a $\bd{K_k}$. For two vertices $u,v \in V(D)$, let $u \sim v$ denote the relation that $uv \not \in A(D)$. We distinguish between two cases and show that both of them lead to a contradiction.

\case{1}{$\sim$ is transitive.}  Then, in particular, the relation of being identical or non-adjacent in $D$ is an equivalence relation. This implies that $D$ is a semicomplete multipartite digraph with parts $I_1,I_2,\ldots,I_\ell$, \emph{i.e.}, $I_j$ is an independent set in $D$ for $j \in \{1,2,\ldots,\ell\}$ and for $u \in I_i, v \in I_j$ with $i \neq j$, $A(D)$ contains at least one of $uv$ and $vu$. Suppose that there are vertices $u \in I_i, v \in I_j$ such that $u$ and $v$ induce a digon in $D$. Then we claim that each pair of vertices $(u',v') \in I_i \times I_j$ induces a digon in $D$. Otherwise, (by symmetry) there are vertices $u,u' \in I_i$ and a vertex $v \in I_j$ such that $\{uv, vu,u'v\} \subseteq A(D)$ but $vu' \not \in A(D)$. But then, $v \sim u', u' \sim u$ and $v \not \sim u$ and hence, $\sim$ is not transitive, a contradiction. This proves the claim that all arcs between $I_i$ and $I_j$ are bidirected. Moreover, by using a similar argumentation, it is easy to see that if there are vertices $u \in I_i, v \in I_j$ with $uv \in A(D)$ and $vu \not \in A(D)$, then $A_D(I_j,I_i)=\varnothing$. Now we claim that $D$ is perfect. By Theorem~\ref{theorem_perfect-digraph} we only need to prove that $D$ does neither contain a filled odd hole, nor a filled odd antihole, nor an induced directed cycle of length at least $3$ as an induced subdigraph.

First assume that $D$ contains a filled odd hole $C$ as an induced subdigraph. Let $v_1,v_2,\ldots,v_r,v_1$ be a cyclic ordering of the vertices of the filled odd hole. By symmetry, we may assume that $v_3 \sim v_1$. As $\sim$ is transitive, this implies that $v_1v_4 \in A(D)$ (as otherwise $v_3 \sim v_1, v_1\sim v_4,$ but $v_3 \not \sim v_4$) and so $v_4 \sim v_1$. As a consequence, $v_1v_3 \in A(D)$ (since $v_4 \sim v_1$ and $v_4v_3 \in A(D)$). By continuing this argumentation we obtain that $v_1v_i \in A(D)$ for all $i \in \{2,\ldots,r\}$. Moreover, regarding $v_2$, it follows that $v_2v_4 \in A(D)$ (as otherwise $v_2 \sim v_4, v_4 \sim v_1,$ but $v_2 \not \sim v_1$, a contradiction). As a consequence, $v_2v_i \in A(D)$ for all $i \in \{4,5,\ldots,r\}$. Finally, $v_3v_r \in A(D)$ (as otherwise $v_3 \sim v_r, v_r \sim v_2$, but $v_3 \not \sim v_2$). However, since $C$ is a filled odd hole, this gives us $v_r \sim v_3$ and so  $v_r \sim v_3, v_3 \sim v_1$, but $v_1 \not \sim v_r$, a contradiction. Thus, $D$ cannot contain a filled odd hole as an induced subdigraph.

Next assume that $D$ contains a filled odd antihole $C$ as an induced subdigraph. Let again $v_1,v_2,\ldots,v_r, v_1$ be a cyclic ordering of the vertices. By symmetry, we may assume that $v_1 \sim v_2$. Then, $v_2v_3 \in A(D)$ as otherwise $\sim$ would not be transitive. Continuing this argument, we obtain that $v_i\sim v_{i+1}$ for $i$ odd and $v_iv_{i+1} \in A(D)$ for even $i$. As $r$ is odd this implies $v_r \sim v_1$. As a consequence, $v_r \sim v_1, v_1 \sim v_2$, but $v_rv_2 \in A(D)$, a contradiction. Thus, $D$ contains no filled antiholes as induced subdigraphs.

Finally, assume that $D$ contains an directed cycle $C$ of length at least $3$ as an induced subdigraph. Again, let $v_1,v_2,\ldots,v_r, v_1$ be a cyclic ordering of the vertices of $C$. Then, $v_1 \sim v_r, v_r \sim v_2$, but $v_1v_2 \in A(D)$, a contradiction. As a consequence, $D$ is perfect by Theorem~\ref{theorem_perfect-digraph} and so $D$ contains a bidirected complete graph of order at least $k$ as a subdigraph and, therefore, a Haj\'os-$k$-constructible sudigraph, which is impossible.

\case{2}{$\sim$ is not transitive.} Then there are vertices $u,v,w \in V(D)$ such that $uv \not \in A(D)$, $vw \not \in A(D)$, but $uw \in A(D)$. By the maximality of $D$, there exist Haj\'os-$k$-constructible subdigraphs $D_{uv} \subseteq D + uv$ and $D_{vw} \subseteq D + vw$. Let $D'$ be the graph obtained from the union $(D_{uv} - uv) \cup (D_{vw} - vw)$ by adding the arc $uw$. Then, $D'$ is a subdigraph of $D$ that can be obtained from disjoint copies of $D_{uv}$ and $D_{uw}$ as follows. First we apply the Haj\'os join by removing the copies of the arcs $uv$ and $vw$, identifying the two copies of $v$, and adding the arc $uw$. Afterwards, for each vertex $x$ that belongs to both $D_{uv}$ and $D_{vw}$, we identify the two copies of $x$. Hence, $D'$ is a Haj\'os-$k$-constructible subdigraph of $D$, a contradiction. This completes the proof.
\end{proofof}

In the last two decades Haj\'os' theorem (Theorem~\ref{theorem_Hajos-constructible}) became very popular among graph theorists. Haj\'os like theorems were established for the list chromatic number by Gravier \cite{Gravier96} and Kr\'al \cite{Kral2004}, for the circular chromatic number by Zhu  \cite{Zhu2003}, for the signed chromatic number by Kang \cite{Kang2018a}, for the chromatic number of edge weighted graphs by Mohar \cite{Mohar2004}, for graph homomorphisms by
Ne\v{s}etril \cite{Nesetril99}, and  for Grassmann homomorphism (a homomorphism concept that
provides a common generalization of graph colorings, hypergraph colorings and
nowhere-zero flows) by Jensen \cite{Jensen2017}.

\section{The Ore construction}
Regarding undirected graphs, Urquhart \cite{Ur97} proved that each graph with chromatic number at least $k$ does not only contain a Haj\'os-$k$-constructible subgraph but itself is Haj\'os-$k$-constructible. Recall that a \textbf{Haj\'os join} of two undirected disjoint graphs $G_1$ and $G_2$ is done by deleting two edges $e=uv \in E(G_1)$  and $e'=u'v' \in E(G_2)$, identifying the vertices $v$ and $v'$, and adding the edge $uu'$. The aim of this section is to point out that the same result does not hold for digraphs and to prove that, however, a slight modification of the Haj\'os join does the trick. The proof of the next theorem is straightforward and left to the reader.

\begin{theorem}
Let $k \geq 3$ be an integer and let $D$ be a Haj\'os-$k$-constructible digraph. Then, $D$ is strongly connected.
\end{theorem}

As a consequence of the above theorem, every digraph with chromatic number at least $k$ that is not strongly connected is not Haj\'os-$k$-constructible and so Urquhart's Theorem cannot be directly transferred to digraphs. Nevertheless, it turns out that we get an Urquhart-type theorem by further allowing the following join. Let $D_1$ and $D_2$ be two digraphs and let $u_1,v_1 \in V(D_1)$ and $u_2,v_2 \in V(D_2)$ such that $D_i[\{u_i,v_i\}]$ is a digon for $i \in \{1,2\}$. Now let $D$ be the digraph obtained from the union $D_1 \cup D_2$ by deleting both arcs between $u_1$ and $v_1$ as well as both arcs between $u_2$ and $v_2$, identifying the vertices $v_1$ and $v_2$ to a new vertex $v$, and adding both arcs $u_1u_2$ and $u_2u_1$. We say that $D$ is the \textbf{bidirected Haj\'os join} of $D_1$ and $D_2$ and write $D= (D_1,v_1,u_1) \bd{\triangledown} (D_2,v_2,u_2)$ or, briefly, $D=D_1 \bd{\triangledown} D_2$. Note that the bidirected Haj\'os join is the exact analogue of the undirected Hajos join. By a slight modification of the proof of Theorem~\ref{theorem_Hajos-join}(a)-(c) one can easily show that the following holds.

\begin{theorem}[Bidirected Haj\'os Construction] \label{theorem_bidirected-Hajos-join}
Let $D=D_1 \bd{\triangledown} D_2$ result from the bidirected Haj\'os join of two disjoint non-empty digraphs $D_1$ and $D_2$. Then, the following statements hold:
\begin{itemize}
\item[\upshape (a)] $\overr{\chi}(D) \geq \min\{\overr{\chi} (D_1), \overr{\chi}(D_2)\}$.
\item[\upshape (b)] If $\overr{\chi}(D_1)=\overr{\chi}(D_2)=k$ and $k \geq 3$, then $\overr{\chi}(D)=k$.
\item[\upshape (c)] If both $D_1$ and $D_2$ are $k$-critical and $k \geq 3$, then $D$ is $k$-critical.
\end{itemize}
\end{theorem}

Note that for the proof  of statement~(b), we use the fact that $k \geq 3$ and so we can choose $\varphi_1$ and $\varphi_2$ such that $\varphi_1(v_1)=\varphi_2(v_2)$ and $\varphi_1(u_1) \neq \varphi_2(u_2)$.
For $k=2$, the statement is not true: for example, $\bd{C_4} \bd{\triangledown} \bd{C_4}= \bd{C_7}$, whereas $\overr{\chi}(\bd{C_4})=2 \neq 3 = \overr{\chi}(\bd{C_7})$.
The same trick works for statement~(c).

For the proof of his Theorem, Urquhart even used  a more restricted class of constructible (undirected) graphs than the class of Haj\'os-$k$-constructible graphs, which originally was introduced by Ore \cite[Chapter 11]{Ore67}. Transferred to digraphs, we get the following. Let $D_1$ and $D_2$ be two vertex-disjoint digraphs, let $u_1v_1$ be an arc of $D_1$, and let $v_2u_2$ be an arc of $D_2$. Furthermore, let $\iota: A_1 \to A_2$ be a bijection with $A_i \subseteq V(G_i - v_i)$ for $i \in \{1,2\}$ and $\iota(u_1) \neq u_2$. Let $D$ be the digraph obtained from $(D_1,v_1,u_1)\triangledown (D_2,v_2,u_2)$ by identifying $w$ with $\iota(w)$ for each $w \in A_1$. Then, $D$ is a \textbf{directed Ore join} of $D_1$ and $D_2$ and we write $D=(D_1,v_1,u_1)\triangledown^o_\iota (D_2,v_2,u_2)$. Similar, if $u_1,v_1 \in V(D_1)$ and $u_2,v_2 \in V(D_2)$ are vertices such that $D_i[\{u_i,v_i\}]$ is a digon for $i \in \{1,2\}$ and if $\iota$ is the bijection from above, then the digraph $D$ obtained from $ (D_1,v_1,u_1) \bd{\triangledown} (D_2,v_2,u_2)$ by identifying  $w$ with $\iota(w)$ for each $w \in A_1$ is a \textbf{bidirected Ore join} of $D_1$ and $D_2$ and we write $D=(D_1,v_1,u_1)\bd{\triangledown}^o_\iota (D_2,v_2,u_2)$.

 We define the class of \textbf{Ore-$k$-constructible} digraphs as the smallest family of digraphs that contains $\bd{K_k}$ and is closed under (directed and bidirected) Ore joins. The proof of Theorem~\ref{theorem_Hajos-constructible} immediately implies the following theorem (see \cite{Ore67} for the undirected analogue). In particular, here we do not need any bidirected Ore joins.

\begin{theorem}
Let $k \geq 3$ be an integer. A digraph has chromatic number at least $k$ if and only if it contains an Ore-$k$-constructible subdigraph.
\end{theorem}

Urquhart proved the following, thereby answering a conjecture by Hanson, Robinson and Toft~\cite{HaRoTo86} (the conjecture was also proposed by Jensen and Toft in their book on graph coloring problems \cite{JeTo95}). Recall that the Ore join of two undirected graphs is done via an undirected Haj\'os join and identification afterwards.

\begin{theorem} Let $k \geq 3$ be an integer. For a graph $G$ the following conditions are equivalent:
\begin{itemize}
\item[\upshape (a)] $G$ is Ore-$k$-constructible.
\item[\upshape (b)] $G$ is Haj\'os-$k$-constructible.
\item[\upshape (c)] The chromatic number of $G$ satisfies $ \chi(G) \geq k$.
\end{itemize}
\end{theorem}

Note that if $G$ is the Haj\'os join of two graphs $G_1$ and $G_2$, then $D(G)$ is the bidirected Haj\'os join of $D(G_1)$ and $D(G_2)$.  Furthermore, $\overr{\chi}(D(G)) = \chi(G)$ and so the above theorem immediately implies the following.

\begin{observation}
 Each bidirected graph with chromatic number at least $k \geq 3$ is Ore-$k$-constructible.
  \label{observation_bidirected}
 \end{observation}

 Now we have all the tools that we need in order to prove our Urquhart-type theorem.

\begin{theorem} \label{theorem_Ore-join}
Let $k \geq 3$ be an integer. A digraph has chromatic number at least $k$ if and only if it is Ore-$k$-constructible.
\end{theorem}

\begin{proof}
It immediately follows from Theorem~\ref{theorem_Hajos-join}(a) and Theorem~\ref{theorem_bidirected-Hajos-join}(a) that each Ore-$k$-constructible digraph has chromatic number at least $k$.

Thus, it suffices to show that each digraph with chromatic number $\geq k$ is Ore-$k$-constructible. We will do this via a sequence of claims. In the following, we will denote by $\bd{K_k} + \overr{v}$ (respectively $\bd{K_k} +  \overl{v}$) the digraph that results from $\bd{K_k}$ by adding a vertex $v$ and the arc $uv$ (respectively $vu$) for some vertex $u$ of the $\bd{K_k}$. Moreover, let $\bd{K_k} + a$ be the digraph that results from $\bd{K_k}$ by adding two new vertices $u,v$ and the arc $a=uv$. Finally, $\mathcal{O}_k$ denotes the class of Ore-$k$-constructible digraphs and $\mathcal{O}_k^*$ denotes the  class of Ore-$k$-constructible digraphs containing a $\bd{K_k}$. It follows from Observation~\ref{observation_bidirected} that
\begin{claim}
The digraph obtained from $\bd{K_k}$ by adding an isolated vertex belongs to $\mathcal{O}_k^*$.
\end{claim}
\begin{claim} \label{claim_K+a}
The digraph $\bd{K_k} + a$ belongs to $\mathcal{O}_k^*$.
\end{claim}

\begin{proof2}
It is clear that $\bd{K_k} + a$ still contains the $\bd{K_k}$. We claim that $\bd{K_k} + a$ is Ore-constructible. To this end, let $D_1$ (respectively $D_2$) be the bidirected graph obtained by identifying a vertex of $\bd{K_k}$ to a vertex of a disjoint $\bd{K_2}$ (respectively $\bd{K_3}$). More formally, $V(D_1)=\{v_1,v_2,\dots,v_k,u\}$, $V(D_2)=\{v'_1,v_2',\dots,v'_k,u_1,u_2\}$, $A(D_1)=\{v_iv_j ~|~ i\neq j\}\cup\{v_1u,uv_1\}$ and $A(D_2)=\{v'_iv'_j ~|~ i\neq j\}\cup\{v'_1u_1,v'_1u_2,u_1v'_1,u_2v'_1,$ $u_1u_2,u_2u_1\}$ (see Figure \ref{figG1G2}). Let $\iota$ be the bijection with $\iota(v_i)=v'_i$ for all $i \in \{1,2,\ldots,k\}$ and let $D'_2=(D_1,u,v_1) \triangledown^o_\iota (D_2,u_2,u_1)$ (see Figure \ref{figure_K_k+a}(a) and (b)).  This Ore-join leads to the digraph $D'_2=D_2-u_2u_1$ (see Figure \ref{figure_K_k+a}(c)). By $v_i^*$ we denote the vertex that results from identifying $v_i$ with $\iota(v_i)=v_i'$.
\begin{figure}[h!]
\vspace{-.3cm}
\begin{subfigure}{\linewidth}
\centering
\includegraphics[width=.5\textwidth]{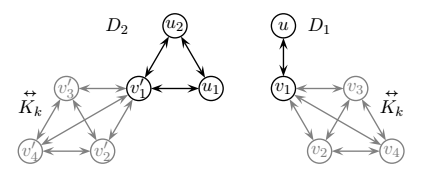}
\caption{The digraphs $D_2$ (on the left) and $D_1$ (on the right).}
\label{figG1G2}
\end{subfigure}
\end{figure}\\

\begin{figure}[h!]\ContinuedFloat
 \begin{subfigure}{\linewidth}
\centering
\includegraphics[width=.5\textwidth]{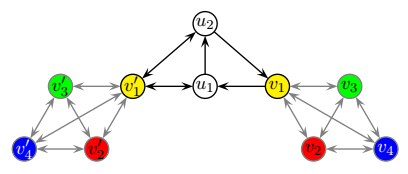}
\caption{The graph we obtain after the directed Hajós join. We then identify the vertices of the same colour.}
\label{figjoin1}
\end{subfigure}\\
\vspace{-.3cm}
 \begin{subfigure}{\linewidth}
\centering
\includegraphics[width=.3\linewidth]{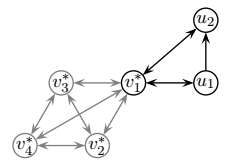}
\caption{This leads to the graph $D'_2$.}
\label{figG'2}
\end{subfigure}
\end{figure}

 Now we take a new copy of $D_1$, define $\iota'$ to be the bijection with $\iota'(v^*_i)=v_{i+1}$ for all $i \in \{1,2,\ldots,k\}$ (where $v_{k+1}=v_1$), and set $D''=(D',u_1,v^*) \bd{\triangledown}^o_{\iota'} (D_1,u,v_1)$ (see Figure \ref{figure_K_k+a}(d)(e)). Still, let $v_i^*$ denote the vertex that results from identifying $v^*_i$ with $\iota'(v^*_i)$.
\begin{figure}[!h]\ContinuedFloat
\begin{subfigure}{0.58\linewidth}
\centering
\includegraphics[width=.7\linewidth]{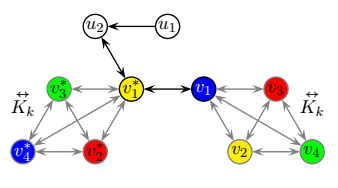}
\caption{}
\label{figG1G'2}
\end{subfigure}
 \begin{subfigure}{0.38\linewidth}
\centering
\includegraphics[width=.6\linewidth]{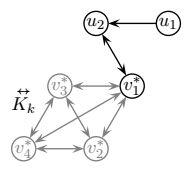}
\caption{The graph $D''_2$.}
\label{figG''2}
\end{subfigure}\\
\end{figure}

 Finally, we take another copy of $G_1$, set $\iota''(v^*_i)=v_{i+1}$ for $i \in \{1,2,\ldots,k\}$ (where $v_{k+1}=v_1$) and perform the Ore join $(D_2'',u_2,v^*)\bd{\triangledown}^o_{\iota''}(G_1,v_1,u)$ (see Figure \ref{figure_K_k+a}(f)(g)). This gives us the digraph $K_k + u_1u_2$ as required.
%

\begin{figure}[!h]\ContinuedFloat
 \begin{subfigure}{0.58\linewidth}
\centering
\includegraphics[width=.7\linewidth]{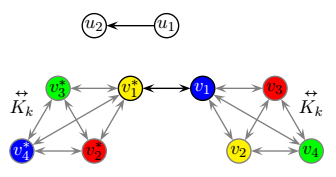} 
\caption{}
\label{figG1G''2}
\end{subfigure}
\begin{subfigure}{0.38\linewidth}
\centering
\includegraphics[width=.6\linewidth]{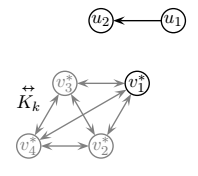} 
\caption{}
\end{subfigure}
\caption{The construction of Claim 2.}
\label{figure_K_k+a}
\end{figure}
\end{proof2}

By using a similar construction starting with the graph that results from $\bd{K_k}$ by adding a vertex $v$ and joining it to two vertices of $\bd{K_k}$ by arcs in both directions we obtain the following. For the exact construction see Figure~\ref{fig_claim_K+v}.

\begin{claim} \label{claim_K+v}
The digraphs $\bd{K_k} + \overr{v}$ and $\bd{K_k} + \overl{v}$ are in $\mathcal{O}_k^*$.
\end{claim}

 \begin{figure}[!h]
 \begin{subfigure}{\linewidth}
\centering
\includegraphics[width=.4\linewidth]{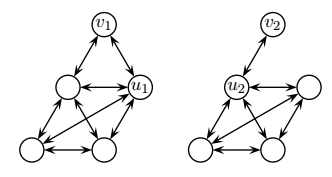}
\caption{We start by performing a directed Haj\'os join between the two depicted graphs.}
\end{subfigure}

\begin{subfigure}{0.58\linewidth}
\centering
\includegraphics[width=.6\linewidth]{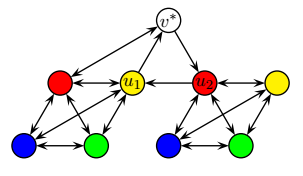}
\caption{Afterwards, we identify the vertices of the same color.}
\end{subfigure}
 \begin{subfigure}{0.38\linewidth}
\centering
\includegraphics[width=.5\linewidth]{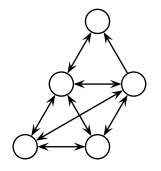}
\caption{End of the first step.}
\end{subfigure}
\end{figure} 

 \begin{figure}[!h]\ContinuedFloat
 \begin{subfigure}{\linewidth}
\centering
\includegraphics[width=.4\linewidth]{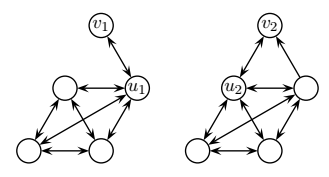}
\caption{Now we perform a bidirected Haj\'os join.}
\end{subfigure}

\begin{subfigure}{0.58\linewidth}
\centering
\includegraphics[width=.6\linewidth]{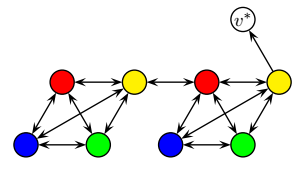}
 \caption{Again we identify vertices of the same color.}
\end{subfigure} 
 \begin{subfigure}{0.38\linewidth}
 \centering
\includegraphics[width=.5\linewidth]{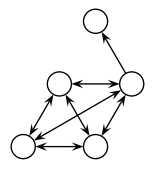}
\caption{This gives us the digraph $K_k+\overr{v}$.}
\end{subfigure}
\caption{The construction of Claim~\ref{claim_K+v}.}
\label{fig_claim_K+v}
\end{figure}

From now on, we may argue similar to the original proof of Urquhart. The next claim can easily be deduced from Claim~1.
\begin{claim}
Let $D$ be a digraph belonging to $\mathcal{O}_k^*$. Then, the digraph $D'$ obtained from $D$ by adding an isolated vertex belongs to $\mathcal{O}_k^*$, too.
\label{claim_isolated-vertex}
\end{claim}

\begin{claim}
Let $D$ be a digraph belonging to $\mathcal{O}_k^*$ and let $a \in A(\overline{D})$. Then, the digraph $D + a$ belongs to $\mathcal{O}_k^*$, too.
\label{claim_new-edge}
\end{claim}
\begin{proof2}
Since $D \in \mathcal{O}_k^*$, there is a vertex set $X \subseteq V(D)$ such that $D[X]$ is a $\bd{K}_k$. We distuingish between two cases.

\smallskip
\noindent \case{1}{One end-vertex of the arc $a$ belongs to $X$.} Then, let $a=uv$ with $u \in X$ and $v \in V \setminus X$ (the case $a=vu$ can be done analogously). Furthermore, let $D'$ be a copy of $\bd{K_k} + \overr{v'}$ and let $u'$ be the vertex adjacent to $v'$ in $D'$.  Finally, let $w, z \in X \setminus\{u\}$ and let $w', z' \in D' \setminus \{u',v'\}$. By Claim~\ref{claim_K+v}, $D' \in \mathcal{O}_k$. Now let $\iota$ be a bijection from $(X\setminus \{u\}) \cup \{v\}$ to $(X'\setminus \{u'\}) \cup \{v'\}$ with $\iota(v) = v'$, $\iota(w)=z'$, and $\iota(z)=w'$. Then, $(D,u,w) \bd {\triangledown^o_\iota} (D',u',w') \in \mathcal{O}_k$  is a copy of $D+a$, and we are done. \\

\smallskip
\noindent\case{2}{No end-vertex of $a$ belongs to $X$.} Then, let $a=uv$,  and $D'$ be a copy of $\bd{K_k} + u'v'$. By Claim~\ref{claim_K+a}, $D'$ belongs to $\mathcal{O}_k$. Now let $x,y,z$ be three vertices from $X$ and let $\{x',y',z'\} \subseteq D' \setminus \{u,v\}$. Finally, let $\iota$ be a bijection from $X \setminus \{x\} \cup \{u,v\}$ to $D' \setminus \{x'\}$ with $\iota(u)=u'$, $\iota(v)=v'$, $\iota(y)=z'$, and $\iota(z)=y'$. Then, $(D,x,y) \bd{\triangledown}^o_\iota (D',x',y')\in \mathcal{O}_k$ is a copy of $D + a$ and the proof of the claim is complete.
\end{proof2}

It follows from Claims~\ref{claim_isolated-vertex} and \ref{claim_new-edge} that each digraph containing a $\bd{K_k}$ belongs to $\mathcal{O}_k^*$. The remaining part of the proof is by reductio ad absurdum. Let $D$ be a maximal counterexample in the sense that $\overr{\chi}(D) \geq k$, $D$ is not Ore-$k$-constructible, and $D$ has maximum number of edges with respect to this property. Then, $D$ does not contain a $\bd{K_k}$ and if $a \in A(\overline{D})$, $D+a$ belongs to $\mathcal{O}_k$. Now we argue as in the proof of Theorem~\ref{theorem_Hajos-constructible}. For two vertices  $u,v \in V(D)$, let $u \sim  v$ denote the relation that $uv \not \in A(D)$. If $\sim$ is transitive we again conclude that $D$ contains a $\bd{K_k}$, a contradiction. Hence, $\sim$ is not transitive and so there are vertices $u,v,w \in V(D)$ with $uv \not \in A(D)$, $vw \not \in A(D)$, but $uw \in A(D)$. Then, both digraphs $D+uv$ as well as $D+vw$ belong to $\mathcal{O}_k$ and $D$ is the Ore join of two disjoint copies of these two digraphs. Thus, $D$ belongs to $\mathcal{O}_k$, a contradiction.
\end{proof}

\section{A Gallai-type theorem for critical digraphs}\label{section_gallai}
\setcounter{claim}{0}
Let $D$ be a $k$-critical digraph. If $v \in V(D)$, then $D-v$ admits a $(k-1)$-coloring and,
since $\overr{\chi}(D)=k$, $v$ must have an out- and an in-neighbor in each color class of such a coloring. Hence, we have $k-1 \leq \min \{d_D^+(v),d_D^-(v)\}$ for every vertex $v \in V(D)$, which gives us a natural way to classify the vertices of $D$. We say that a vertex $v \in V(D)$ is a \textbf{low-vertex} of $D$ if $d_D^+(v)=d_D^-(v)=k-1$ and a \textbf{high vertex}of $D$, otherwise. Furthermore, let $D_L$ denote the digraph that is induced by the set of low vertices of $D$; we will call it the \textbf{low vertex subdigraph} of $D$. For undirected graphs, Gallai~\cite{Gal63a} proved that the low vertex subgraph has a specific structure. The next theorem transfers his result to digraphs.

\begin{theorem}\label{theorem_gallai}
Let $D_L$ be the low vertex subdigraph of a $k$-critical digraph $D$. Then, each block $B$ of $D_L$ satisfies one of the following statements.
\begin{itemize}
\item[\upshape (a)] $B$ consists of just one single arc.
\item[\upshape (b)] $B$ is a directed cycle of length $\geq 2$.
\item[\upshape (c)] $B$ is a bidirected cycle of odd length.
\item[\upshape (d)] $B$ is a complete bidirected graph.
 \end{itemize}
\end{theorem}

 For the proof of Theorem~\ref{theorem_gallai} we will use a theorem of  Harutyunyan and Mohar~\cite{HaMo11} concerning list-colorings of digraphs. Given a digraph $D$, a \textbf{list-assignment} $L$ is a function that assigns each vertex $v \in V(D)$ a set (list) of colors. An $L$\textbf{-coloring} $\varphi$ of $D$ is a coloring of $D$ such that $\varphi(v) \in L(v)$ for all $v \in V(D)$.

 \begin{theorem}[\cite{HaMo11}] \label{theorem_harut}
Let $D$ be a connected digraph, and let $L$ be a list-assignment such that $|L(v)| \geq \max \{d_D^+(v), d_D^-(v)\}$ for all $v \in V(D)$. Suppose that $D$ is not $L$-colorable. Then, $D$ is Eulerian and for every block $B$ of $D$ one of the following cases occurs:
\begin{itemize}
\item[\upshape (a)] $B$ is a directed cycle of length $\geq 2$.
\item[\upshape (b)] $B$ is a bidirected cycle of odd length $\geq 3$.
\item[\upshape (c)] $B$ is a bidirected complete graph.
\end{itemize}
 \end{theorem}

The next proposition states some important facts that will be needed for the proof of Theorem~\ref{theorem_gallai}.

 \begin{proposition}\label{prop_shifting}
Let $D_L$ be the low vertex subdigraph of a $k$-critical digraph $D$.  Moreover, given a vertex $v \in V(D_L)$, let $\varphi$ be a $(k-1)$-coloring of $D - v$ with color set $C$. Then the following statements hold:
\begin{itemize}
\item[\upshape (a)] Each color from $C$ appears exactly once in $N^+(v)$ and in $N^-(v)$.
\item[\upshape (b)] If $u \in V(D_L)$ is adjacent to $v$, then uncoloring $u$ and coloring $v$ with the color of $u$ leads to a $(k-1)$-coloring of $D-u$.
\end{itemize}
 \end{proposition}

 \begin{proof}
 Suppose (by symmetry) that there is a color $\alpha \in C$ such that $\alpha$ does not appear in $N^+(v)$. Then, coloring $v$ with $\alpha$ cannot create a monochromatic cycle in $D$ (as $v$ has no out-neighbor with color $\alpha$) and, thus, $D$ would be $(k-1)$-colorable, a contradiction. As $d_D^+(v)=k-1=|C|$, this proves (a).

For the proof of (b), assume (by symmetry) that $uv \in A(D)$. Then it follows from (a) that after uncoloring $u$, $v$ has no in-neighbor with color $\varphi(u)$ and so coloring $v$ with color $\varphi(u)$ cannot create a monochromatic cycle.
 \end{proof}

 In the following, we will call the procedure that is described in Proposition~\ref{prop_shifting} (b) \textbf{shifting} the color from $u$ to $v$ and briefly write $u \to v$. Now let $D$ be a $k$-critical digraph, let $C$ be a (not necessarily directed) cycle in $D_L$ and let $v \in V(C)$. Moreover, let $\varphi$ be a $(k-1)$ coloring of $D - v$ and let $u$ and $w$ be the vertices such that $u,v$ and $w$ are consecutive in $C$. Then, beginning with $u \to v$, we can shift each vertex of $C$, one after another, clockwise and obtain a new $(k-1)$-coloring of $D - v$ (see Figure~\ref{fig_shift-example}). Similar, beginning with $w \to v$, we can shift each vertex of $C$ counter-clockwise and obtain a third $(k-1)$-coloring of $D-v$. The main idea for this goes back to Gallai~\cite{Gal63a}; we will use this observation frequently in the following.

 \begin{figure}[htbp]
\centering
\begin{tikzpicture}[>={[scale=1.1]Stealth}]

\node[draw=none,minimum size=3cm,regular polygon,regular polygon sides=5] (a) {};

\node[vertex, fill=red, label={above:$w$}] (a1) at (a.corner 1) {};
\node[vertex, label={left:$v$}] (a2) at (a.corner 2){};
\node[vertex, fill=blue, label={left:$u$}] (a3) at (a.corner 3){};
\node[vertex, fill=red, label={right:$v_4$}] (a4) at (a.corner 4){};
\node[vertex, fill=green, label={right:$v_3$}] (a5) at (a.corner 5){};
\node[circle, inner sep=2pt, draw=black] at (a1) [xshift=-1.5cm]{$1$};

\path[-]
(a1) edge [->] (a2)
(a2) edge [<-] (a3)
(a3) edge [->] (a4)
(a4) edge [->] (a5)
(a1) edge [->] (a5);

\node (h1) at (a1) [xshift=2.5cm]{};
\node (h2) at (h1) [yshift=-3cm]{};
\draw[dashed] (h1) -- (h2);

\begin{scope}[xshift=5cm]{};
\node[draw=none,minimum size=3cm,regular polygon,regular polygon sides=5] (a) {};

\node[vertex, fill=red, label={above:$w$}] (a1) at (a.corner 1) {};
\node[vertex, fill= blue, label={left:$v$}] (a2) at (a.corner 2){};
\node[vertex, label={left:$u$}] (a3) at (a.corner 3){};
\node[vertex, fill=red, label={right:$v_4$}] (a4) at (a.corner 4){};
\node[vertex, fill=green, label={right:$v_3$}] (a5) at (a.corner 5){};

\path[-]
(a1) edge [->] (a2)
(a2) edge [<-] (a3)
(a3) edge [->] (a4)
(a4) edge [->] (a5)
(a1) edge [->] (a5);

\node (h1) at (a1) [xshift=2.5cm]{};
\node (h2) at (h1) [yshift=-3cm]{};
\draw[dashed] (h1) -- (h2);
\node (l1) at (a1) [yshift=-3.5cm] {$u \to v$};
\node[circle, inner sep=2pt, draw=black] at (a1) [xshift=-1.5cm]{$2$};
\end{scope}

\begin{scope}[xshift=10cm]{};
\node[draw=none,minimum size=3cm,regular polygon,regular polygon sides=5] (a) {};

\node[vertex, fill=red, label={above:$w$}] (a1) at (a.corner 1) {};
\node[vertex, fill= blue, label={left:$v$}] (a2) at (a.corner 2){};
\node[vertex, fill=red, label={left:$u$}] (a3) at (a.corner 3){};
\node[vertex, label={right:$v_4$}] (a4) at (a.corner 4){};
\node[vertex, fill=green, label={right:$v_3$}] (a5) at (a.corner 5){};

\path[-]
(a1) edge [->] (a2)
(a2) edge [<-] (a3)
(a3) edge [->] (a4)
(a4) edge [->] (a5)
(a1) edge [->] (a5);

\node (h1) at (a1) [xshift=2.5cm]{};
\node (h2) at (h1) [yshift=-3cm]{};
\node (l1) at (a1) [yshift=-3.5cm] {$v_4 \to u$};
\node[circle, inner sep=2pt, draw=black] at (a1) [xshift=-1.5cm]{$3$};
\end{scope}

\begin{scope}[yshift=-4.5cm]{};
\node[draw=none,minimum size=3cm,regular polygon,regular polygon sides=5] (a) {};

\node[vertex, fill=red, label={above:$w$}] (a1) at (a.corner 1) {};
\node[vertex, fill= blue, label={left:$v$}] (a2) at (a.corner 2){};
\node[vertex, fill=red, label={left:$u$}] (a3) at (a.corner 3){};
\node[vertex, fill=green, label={right:$v_4$}] (a4) at (a.corner 4){};
\node[vertex, label={right:$v_3$}] (a5) at (a.corner 5){};

\path[-]
(a1) edge [->] (a2)
(a2) edge [<-] (a3)
(a3) edge [->] (a4)
(a4) edge [->] (a5)
(a1) edge [->] (a5);

\node (h1) at (a1) [xshift=2.5cm]{};
\node (h2) at (h1) [yshift=-3cm]{};
\draw[dashed] (h1) -- (h2);
\node (l1) at (a1) [yshift=-3.5cm] {$v_3 \to v_4$};
\node[circle, inner sep=2pt, draw=black] at (a1) [xshift=-1.5cm]{$4$};
\end{scope}

\begin{scope}[yshift=-4.5cm, xshift=5cm]{};
\node[draw=none,minimum size=3cm,regular polygon,regular polygon sides=5] (a) {};

\node[vertex, label={above:$w$}] (a1) at (a.corner 1) {};
\node[vertex, fill= blue, label={left:$v$}] (a2) at (a.corner 2){};
\node[vertex, fill=red, label={left:$u$}] (a3) at (a.corner 3){};
\node[vertex, fill=green, label={right:$v_4$}] (a4) at (a.corner 4){};
\node[vertex, fill=red, label={right:$v_3$}] (a5) at (a.corner 5){};

\path[-]
(a1) edge [->] (a2)
(a2) edge [<-] (a3)
(a3) edge [->] (a4)
(a4) edge [->] (a5)
(a1) edge [->] (a5);

\node (h1) at (a1) [xshift=2.5cm]{};
\node (h2) at (h1) [yshift=-3cm]{};
\draw[dashed] (h1) -- (h2);
\node (l1) at (a1) [yshift=-3.5cm] {$w \to v_3$};
\node[circle, inner sep=2pt, draw=black] at (a1) [xshift=-1.5cm]{$5$};
\end{scope}

\begin{scope}[yshift=-4.5cm, xshift=10cm]{};
\node[draw=none,minimum size=3cm,regular polygon,regular polygon sides=5] (a) {};

\node[vertex, fill=blue, label={above:$w$}] (a1) at (a.corner 1) {};
\node[vertex, label={left:$v$}] (a2) at (a.corner 2){};
\node[vertex, fill=red, label={left:$u$}] (a3) at (a.corner 3){};
\node[vertex, fill=green, label={right:$v_4$}] (a4) at (a.corner 4){};
\node[vertex, fill=red, label={right:$v_3$}] (a5) at (a.corner 5){};

\path[-]
(a1) edge [->] (a2)
(a2) edge [<-] (a3)
(a3) edge [->] (a4)
(a4) edge [->] (a5)
(a1) edge [->] (a5);

\node (h1) at (a1) [xshift=2.5cm]{};
\node (h2) at (h1) [yshift=-3cm]{};
\node (l1) at (a1) [yshift=-3.5cm] {$v \to w$};
\node[circle, inner sep=2pt, draw=black] at (a1) [xshift=-1.5cm]{$6$};
\end{scope}

\end{tikzpicture}
\caption{The black vertex denotes the clockwise shifting around a cycle.}
\label{fig_shift-example}
\end{figure}
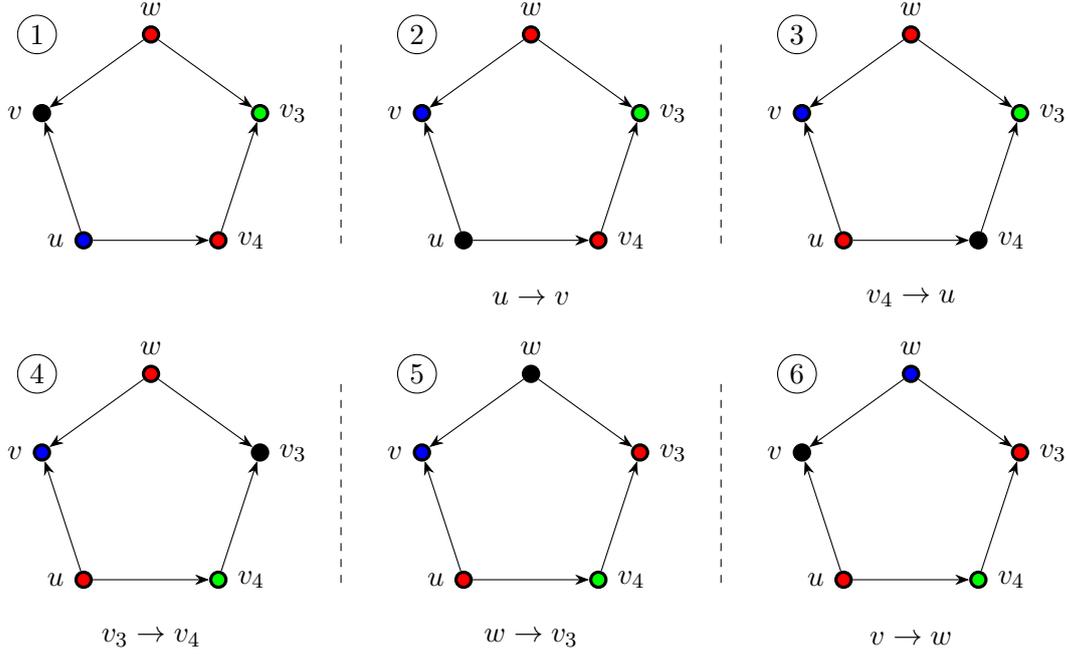

 \begin{proofof}[Theorem~\ref{theorem_gallai}]
Let $D_L$ be the low vertex subdigraph of a $k$-critical digraph $D$ and let $B$ be an arbitrary block of $D_L$. If $|B|=1$, then $B=\bd{K_1}$ and we are done. If $|B|=2$, then either $B$ consists of just one arc or $B$ is a bidirected complete graph and so there is nothing to show. Thus, we may assume $|B| \geq 3$.

 \begin{claim} \label{claim_eulerian}
 For all vertices $v \in V(B)$ we have $d_{B}^+(v) = d_{B}^-(v)$ and so $B$ is Eulerian.
 \end{claim}
 \begin{proof2}
 For otherwise, we may assume that $d_B^+(v) < d_B^-(v)$ for some $v \in V(B)$. Let $\varphi$ be a $(k-1)$-coloring of $D-v$. Since $d_D^+(v)=d_D^-(v) = k-1$, it follows from Proposition~\ref{prop_shifting}(a) that there is a color $\alpha$ that appears in $N_B^-(v)$ but not in $N_B^+(v)$. Let $u$ be the vertex from $N_B^-(v)$ with $\varphi(u)=\alpha$. Note that Proposition~\ref{prop_shifting}(a) furthermore implies that there is a vertex in $v' \in N_D^+(v) \cap (V(D) \setminus V(B))$ that has color $\alpha$. First we show that $d_B^+(v)=0$.  Suppose, to the contrary, that $d^+_B(v) > 0$ and let $w$ be an out-neighbor of $v$ in $D_L$. Then, in $B$ there is a (not-necessarily directed) induced cycle $C$ such that $u,v$ and $w$ are consecutive on $C$. Beginning with $u \to v$, we shift all vertices of $C$ clockwise and obtain a new $(k-1)$-coloring $\varphi'$ of $D-v$ with $\varphi'(w)=\alpha$. Since no vertex from $V(D) \setminus V(C)$ took part in the shifting, we have $\varphi'(v')=\varphi(v')=\alpha$ and so $\alpha$ appears twice in $N_D^+(v)$, contradicting Proposition~\ref{prop_shifting}(a). This proves that $d_B^+(v)=0$.

Let again $C$ be an (undirected) induced cycle in $B$ such that $u$ and $v$ are consecutive on $C$ and let $w$ be the other neighbor of $v$ in $C$. Then, $w$ is also an in-neighbor of $v$ (as $d_B^+(v)=0$). Thus, it follows from Proposition~\ref{prop_shifting}(a) that $\varphi(w)\neq \varphi(u)$, say $\varphi(w)=\beta$. Moreover, we obtain that the vertices of $C$ (except from $v$) are colored alternately with $\beta$ and $\alpha$. Otherwise, there are two consecutive vertices $x,x'$ on $C$ such that $\{\varphi(x), \varphi(x')\} \neq \{\alpha, \beta\}$. Then we can shift the colors around the vertices of $C$ such that $u$ gets color $\varphi(x)$ and $w$ gets color $\varphi(x')$ and obtain a $(k-1)$-coloring $\varphi'$ of $D-v$ with $\{\varphi'(u), \varphi'(w)\} \neq \{\alpha,\beta\}$, which contradicts Proposition~\ref{prop_shifting}(a) as $C$ is induced and so no neighbors of $v$ besides $u$ and $w$ take part in the shifting.

 As a consequence, $C$ has odd length. Now let $v=v_1,w=v_2,v_3,\ldots,u=v_r,v_1$ be a cyclic ordering of the vertices of $C$. We claim that $v_3v_2 \not \in A(D)$. Assume, to the contrary, $v_3v_2 \in A(D)$. Then, we can shift $w \to v$ and obtain a coloring $\varphi'$ of $D-w$ with $\varphi'(v)=\beta$ and $\varphi'(v_3)= \alpha$. In particular, $v_3$ is the only in-neighbor of $w$ that has color $\alpha$ with respect to $\varphi'$. On the other hand, beginning from $\varphi$ with $u \to v$, we can shift every vertex besides $v$ clockwise around $C$  (the last shift is $w \to v_3$) and get a $(k-1)$-coloring $\varphi^*$ of $G-w$ with $\varphi^*(v) = \alpha$ and $\varphi^*(v_3)= \beta$. As $vw \not \in A(D)$ and as $C$ is induced, it follows that $w$ has no in-neighbor that has color $\alpha$ with respect to $\varphi^*$, a contradiction. Hence, $v_3v_2 \in A(D)$ and so $v_2v_3 \in A(D)$. By repeating this argumentation, we obtain that $v_{i+1}v_i \not \in A(D)$ but $v_iv_{i+1} \in A(D)$ for $i \geq 2$ even and that $v_iv_{i+1} \not \in A(D)$ but $v_{i+1}v_i \in A(D)$ for $i \geq 3$ odd. In particular, this leads to $v_rv \not \in  A(D)$, a contradiction. This proves the claim.
 \end{proof2}
Now let $\varphi$ be a $(k-1)$-coloring of $D-B$ with color set $C$. For $v \in V(B)$, let
$$L(v) = C   \setminus \varphi(N_D^+(v) \setminus V(B)).$$
then, as $d_D^+(v)=d_D^-(v)=k-1=|C|$ and since $d_B^+(v)=d_B^-(v)$ by Claim~\ref{claim_eulerian}, we have $|L(v)| \geq \max \{d_B^+(v), d_B^-(v)\}$ for all $v \in V(B)$. Moreover, $B$ is not $L$-colorable, as the union of any $L$-coloring of $B$ with $\varphi$ would clearly lead to a $(k-1)$-coloring of $D$. Hence, we can apply Theorem~\ref{theorem_harut} and so $B$ is a directed cycle, or an odd bidirected cycle, or a bidrected complete graph, as claimed.
 \end{proofof}

In the undirected case, Gallai~\cite{Gal63a} showed that the only blocks of the low vertex graph are complete graphs or odd cycles. Although for digraphs the directed cycles arise naturally, it may surprise that there can also be blocks  that consist of just one arc. That this indeed may happen is illustrated in Figure~\ref{fig_K_4Haj}, where we show the Haj\'os join of two $\bd{K_4}$; here the low vertex subdigraph consists of every vertex except the identified vertex $v$. Clearly, by starting from a $\bd{K_k}$ and iteratively taking Haj\'os joins with another $\bd{K_k}$, we can even create infinite families of digraphs $D$ such that there are blocks of $D_L$ consisting of just a single arc.

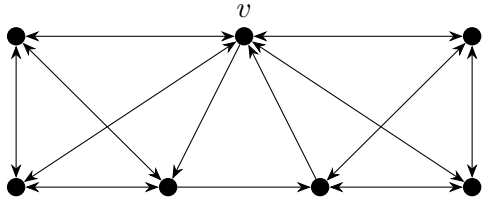
\begin{figure}[htbp]
\centering
\begin{tikzpicture}[>={[scale=1.1]Stealth}]

\node[vertex] (u1){};
\node[vertex, label={above:$v$}, xshift=3cm] (u2) at (u1) {};
\node[vertex, xshift=3cm] (u3) at (u2) {};
\node[vertex,  yshift=-2cm] (u4) at (u1) {};
\node[vertex, xshift=2cm] (u5) at (u4) {};
\node[vertex, xshift=2cm] (u6) at (u5){};
\node[vertex, xshift=2cm] (u7) at (u6){};

\path[-]
(u1) edge [<->] (u2)
	    edge [<->] (u4)
	    edge [<->] (u5)
(u2) edge [<->] (u4)
		edge [->] (u5)
		edge [<-] (u6)
		edge [<->] (u3)
		edge [<->] (u7)
(u3) edge [<->] (u6)
	    edge [<->] (u7)
(u4) edge [<->] (u5)
(u5) edge [->] (u6)
(u6) edge [<->] (u7);		;	    

\end{tikzpicture}
\caption{The Haj\'{o}s join of two bidirected $K_4$.}
\label{fig_K_4Haj}
\end{figure}

Gallai used the characterization of the low vertex subgraph of critical graphs he obtained in \cite{Gal63a} to establish a lower bound for the number of edges of critical graphs. We can apply the same approach to obtain a similar bound for the number of arcs in critical digraphs, see also \cite{SchwST2017}.

\begin{corollary}
\label{corollary:gallaibound}
Let $D$ be a $(k+1)$-critical digraph with $k\geq 3$ and without digons. Then
$$2|A(D)|\geq \left( 2k+\frac{2k-2}{(2k+1)^2-3} \right)|D|$$
unless $D=\bd{K}_{k+1}$.
\end{corollary}

\begin{proof}
Let $V=V(D)$ and let $n=|V|$. For a set $X\subseteq V$, let $a(X)$ denote the number of arcs of $D[X]$. Furthermore, let
$$R=\left(2k+\frac{2k-2}{(2k+1)^2-3} \right).$$
Our aim is to show that $2a(V)\geq Rn$. Let $U=V(D_L)$ be the set of low vertices of $D$ and let $W=V\sm U$. If $U=\ems$, then $2a(V)\geq (2k+1)n\geq Rn$ and we are done. So assume that $U\not=\ems$. Since $D$ has no digons and $D\not=\bd{K}_{k+1}$, it follows that $\bd{K}_{k+1}$ is no subdigraph of $D$. By Theorem~\ref{theorem_gallai}, each block of $D_L$ consists of exactly one arc or is a directed cycle of length $\geq 3$. Then, see \cite[Lemma 3.3]{SchwST2017}, we have
$$\left( 2k-1+\frac{1}{k}\right)|U|-2a(U)\geq 2.$$
Since every vertex of $U$ has total degree $2k$ in $D$ (i.e., $d_D^+(v)+d_D^-(v)= 2k$ for all $v\in U$) and $n=|U|+|W|$, we then obtain that
$$2a(V)=2a(W)+4k|U|-2a(U)\geq 4k|U|-2a(U)\geq \left(2k+1-\frac{1}{k}\right)|U|+2$$
On the other hand, since every vertex in $W$ has total degree at least $2k+1$, we obtain that
$$2a(V)\geq 2kn+|W|\geq (2k+1)n-|U|.$$
Adding the first inequality to the second inequality multiplied with $(2k+1-1/k)$ yields
$$2a(V)(2k+2-1/k)\geq (2k+1-1/k)(2k+1)n+2.$$
As $(2k+2-1/k)=((2k)^2+4k-2)/(2k)>0$, this leads to
$$2a(V)\geq \frac{((2k)^2+2k-2)(2k+1)n+4k}{(2k)^2+4k-2}\geq Rn.$$
Thus the proof is complete.
\end{proof}

\section{Open Questions}
Since the field of critical digraphs is still wide open, a lot of questions immediately come to mind. It follows from Theorem~\ref{theorem_Hajos-constructible} that each $k$-critical digraph is Haj\'os-$k$-constructible. However, the proof of Theorem~\ref{theorem_Hajos-constructible} is not constructive at all and the authors feel quite embarrassed in admitting that they could not even manage to construct a bidirected cycle of length five from $\bd{K_4}$'s using Haj\'os joins and identification of non-adjacent vertices. Thus, we want to pose the following question.

\begin{question}
How can a bidirected $C_5$ be constructed from $\bd{K_4}$'s by only using Haj\'os joins and identifying non-adjacent vertices?
\end{question}

Building upon this question, it is of particular interest to study the connection of the Haj\'os construction to computational complexity. In the undirected case, Mansfield and Welsh~\cite{MaWe82} stated the problem of determining the complexity of the Haj\'os construction. They noted that if for any $k \geq 3$ there would exist a polynomial $P$ such that every graph of order $n$ with chromatic number $k$ contains a Haj\'os-$k$-constructible subgraph that can be obtained by at most $P(n)$ uses of the Haj\'os-join and identification of non-adjacent vertices, then $\NP = \coNP$.  Hence, it is very likely that the Haj\'os construction is not polynomially bounded but not much progress has been made on this problem yet. Pitassi and Urquhart~\cite{PiUr95} found a linkage to another important open problem in logic; they proved that a restricted version of the Haj\'os construction is polynomially bounded if and only if extended Frege systems are polynomially bounded.

\begin{question}
For $k \geq 3$, is there a polynomial $P$ such that every digraph of order $n$ contains a Haj\'os-$k$-constructible subdigraph that can be obtained from $\bd{K_k}$'s by at most $P(n)$ uses of the Haj\'os-join and identification of non-adjacent vertices?
\end{question}

A beautiful theorem of Gallai~\cite{Gal63b} states that any $k$-critical graph with order at most $2k-2$ and $k \geq 2$ is the Dirac join of two disjoint non-empty critical graphs (for the Dirac join of two undirected graphs just add all possible edges between the two graphs $G_1$ and $G_2$). Within the last decades, various different proofs of this theorem have been published (see e.g. \cite{Mol99} and \cite{Ste03}). Clearly, a graph $G$ is the Dirac join of two disjoint non-empty graphs if and only if $\overline{G}$ is disconnected and so most of the proofs use matching theory for the complement graph $\overline{G}$. However, it is yet unclear how to do this for digraphs.

\begin{question}
Let $k \geq 3$ be an integer. Is there a $k$-critical digraph $D$ on at most $2k - 2$ vertices that is not the Dirac join of two proper digraphs $D_1$ and $D_2$?
\end{question}

In coloring theory of digraphs, it is often of particular interest how digon-free digraphs behave. For example, it was shown by Harutyunyan in his PhD thesis~\cite{Ha11} that almost all tournaments of order $n$ have chromatic number at least $\frac{1}{2}(\frac{n}{\log n + 1})$. As a consequence, if $n$ is large enough then for some $k \geq \frac{1}{2}(\frac{n}{\log n + 1})$ there are $k$-critical digon-free digraphs on at most $n$ vertices. This leads to our final question.

\begin{question}
For fixed $k \geq 3$, what is the minimum integer $N(k)$ such that there is a $k$-critical digon-free digraph on $N(k)$ vertices?
\end{question}

As $k-1 \leq \min \{ d_D^+(v), d_D^-(v)\}$ for all vertices $v$ of a $k$-critical digraph $D$, we trivially have $N(k) \geq 2k-1$. In fact, Brooks' theorem for digraphs~\cite{Mo10} implies that $N(k) \geq 2k$ for $k\geq 3$. Moreover, some small values are already known: the directed triangle shows that $N(2) = 3$, and Neumann-Lara~\cite{NeuLa94} proved that $N(3) = 7$, $N(4) = 11$, and $17 \leq N(5) \leq 19$; he conjectured that $N(5) = 17$.

\end{document}